\documentclass[12pt,a4paper]{article}
\usepackage[latin9]{inputenc}
\usepackage[active]{srcltx}
\usepackage{amsmath}
\usepackage{amssymb}

\makeatletter

\newcommand{\lyxaddress}[1]{
\par {\raggedright #1
\vspace{1.4em}
\noindent\par}
}

\@ifundefined{date}{}{\date{}}

\usepackage{amsfonts}
\usepackage{amsthm}
\usepackage{mathrsfs}

\newtheorem{theorem}{Theorem}
\newtheorem{proposition}[theorem]{Proposition}
\newtheorem{lemma}[theorem]{Lemma}

\newtheorem*{lemma*}{Lemma}

\theoremstyle{remark}
\newtheorem{remark}[theorem]{Remark}
\newtheorem*{remark*}{Remark}
\newtheorem*{remarks*}{Remarks}
\newtheorem*{example*}{Example}
\newtheorem*{question*}{QUESTION}
\newtheorem*{conjecture*}{CONJECTURE}

\theoremstyle{definition}
\newtheorem{definition}[theorem]{Definition}
\newtheorem*{definition*}{Definition}

\newtheorem*{notation*}{Notation}

\newcommand{\Dom}{\mathop\mathrm{Dom}\nolimits}
\newcommand{\Spec}{\mathop\mathrm{spec}\nolimits}
\newcommand{\supp}{\mathop\mathrm{supp}\nolimits}
\renewcommand{\span}{\mathop\mathrm{span}\nolimits}
\newcommand{\sgn}{\mathop\mathrm{sgn}\nolimits}

\newcommand{\Tr}{\mathop\mathrm{tr}\nolimits}

\makeatother

\begin{document}

\title{On infinite Jacobi matrices with a trace class resolvent}

\author{Pavel \v{S}\v{t}ov\'\i\v{c}ek}
\maketitle

\lyxaddress{Department of Mathematics, Faculty of Nuclear Science, Czech Technical
University in Prague, Trojanova 13, 120~00 Praha, Czech Republic}
\begin{abstract}
\noindent Let $\{\hat{P}_{n}(x)\}$ be an orthonormal polynomial sequence
and denote by $\{w_{n}(x)\}$ the respective sequence of functions
of the second kind. Suppose the Hamburger moment problem for $\{\hat{P}_{n}(x)\}$
is determinate and denote by $J$ the corresponding Jacobi matrix
operator on $\ell^{2}$. We show that if $J$ is positive definite
and $J^{-1}$ belongs to the trace class then the series on the right-hand
side of the defining equation
\[
\mathfrak{F}(z):=1-z\sum_{n=0}^{\infty}w_{n}(0)\hat{P}_{n}(z)
\]
converges locally uniformly on $\mathbb{C}$ and it holds true that
$\mathfrak{F}(z)=\prod_{n=1}^{\infty}(1-z/\lambda_{n})$ where $\{\lambda_{n};\,n=1,2,3,\ldots\}=\Spec J$.
Furthermore, the Al-Salam\textendash Carlitz II polynomials are treated
as an example of orthogonal polynomials to which this theorem can
be applied.
\end{abstract}
\noindent \begin{flushleft}
\emph{Keywords}: infinite Jacobi matrix; orthogonal polynomials; functions
of the second kind; spectral problem; Al-Salam\textendash Carlitz
II polynomials \emph{}\\
\emph{MSC codes}: 42C05; 47B36
\par\end{flushleft}

\section{Introduction}

Denote for a while by $\mathfrak{A}_{\infty}\subset\mathbb{C}^{\infty,\infty}$
the unital algebra whose elements are semi-infinite matrices $\mathcal{A}=(\mathcal{A}_{m,n})_{m,n\geq0}$
obeying the condition
\[
\exists k=k_{\mathcal{A}}\in\mathbb{Z}_{+}\ \,\text{s.t.}\,\ \forall m,n\in\mathbb{Z}_{+},\ \mathcal{A}_{m,n}=0\ \text{if}\ n>m+k.
\]
Here $\mathbb{Z}_{+}$ stands for non-negative integers. Clearly,
if $\mathcal{A},\mathcal{B}\in\mathfrak{A}_{\infty}$ and $\mathcal{A}_{m,n}=0$
for $n>m+k$, $\mathcal{B}_{m,n}=0$ for $n>m+\ell$ then $(\mathcal{A}\mathcal{B})_{m,n}=0$
for $n>m+k+\ell$. The product $\mathcal{A}\mathcal{B}$ makes good
sense even for any $\mathcal{A}\in\mathfrak{A}_{\infty}$ and $\mathcal{B}\in\mathbb{C}^{\infty,\infty}$.

Regarding $\mathbb{C}^{\infty}$ as a linear space of semi-infinite
column vectors every $\mathcal{A}\in\mathfrak{A}_{\infty}$ becomes
naturally a linear operator on $\mathbb{C}^{\infty}$. The Hilbert
space $\mathbb{\ell}^{2}\equiv\ell^{2}(\mathbb{Z}_{+})$ may be viewed
as a linear subspace of $\mathbb{C}^{\infty}$. Denote by $\{\mathbf{e}_{n}\}$
the canonical basis in $\ell^{2}$. If every column of $\mathcal{A}\in\mathfrak{A}_{\infty}$
is square summable then $\mathcal{A}$ can be regarded as well as
an operator in $\mathbb{\ell}^{2}$ with $\Dom\mathcal{A}:=\span\{\mathbf{e}_{n}\}$.
In this case we say that $\mathbf{f}\in\mathbb{C}^{\infty}$ is a
\emph{formal eigenvector} of $\mathcal{A}$ if it is an eigenvector
of $\mathcal{A}$ in $\mathbb{C}^{\infty}$.

Furthermore, we shall extend the usual scalar product $\langle\cdot,\cdot\rangle$
on $\ell^{2}$ to $\mathbb{C}^{\infty}$ by letting $\langle\mathbf{f},\mathbf{g}\rangle:=\sum_{n=0}^{\infty}\overline{f_{n}}g_{n}$
for any two vectors $\mathbf{f},\mathbf{g}\in\mathbb{C}^{\infty}$
such that the series converges. Particularly this is the case if one
of the vectors has only finitely many non-vanishing components.

This point of view is applicable to any Jacobi (tridiagonal) matrix
which is in the sequel always supposed to be semi-infinite, symmetric,
real and non-decomposable. A Jacobi matrix $\mathcal{J}$ will be
written in the form
\[
\mathcal{J}=\left(\begin{array}{cccccc}
\beta_{0} & \alpha_{0} & 0 & 0 & 0 & \cdots\\
\alpha_{0} & \beta_{1} & \alpha_{1} & 0 & 0 & \cdots\\
0 & \alpha_{1} & \beta_{2} & \alpha_{2} & 0 & \cdots\\
0 & 0 & \alpha_{2} & \beta_{3} & \alpha_{3} & \cdots\\
0 & 0 & 0 & \alpha_{3} & \beta_{4} & \cdots\\
\vdots & \vdots & \vdots & \vdots & \vdots & \ddots
\end{array}\right)\!,
\]
Without loss of generality we can even suppose that $\alpha_{n}>0$
for all $n$.

If $\mathcal{A}\in\mathfrak{A}_{\infty}$ regarded as an operator
in $\ell^{2}$ is bounded it extends as a bounded operator unambiguously
to the whole Hilbert space and therefore, with some abuse of notation,
we can use the same symbol for a bounded operator on $\ell^{2}$ and
its matrix (expressed in the canonical basis). In particular, the
symbol $I$ will denote, at the same time, the unit operator on $\ell^{2}$
and the unit matrix in $\mathbb{C}^{\infty,\infty}$. With a general
Jacobi matrix one has to be more careful, however.

Recall that with every Jacobi matrix $\mathcal{J}$ one associates
a sequence of moments
\begin{equation}
m_{k}:=\langle\mathbf{e}_{0},\mathcal{J}^{k}\mathbf{e}_{0}\rangle,\ k\geq0.\label{eq:J_moments}
\end{equation}
The following definition may simplify some formulations.

\begin{definition} We shall say that a Jacobi matrix $\mathcal{J}$
is Hamburger determinate if the Hamburger moment problem for the sequence
of moments (\ref{eq:J_moments}) is determinate. \end{definition}

Given a Jacobi matrix $\mathcal{J}$ let us re-denote $\mathcal{J}$
as $\dot{J}$ if treated as an operator on $\ell^{2}$. Hence $\Dom\dot{J}:=\span\{\mathbf{e}_{n}\}$.
Clearly, $\dot{J}$ is symmetric. As is well known, $\mathcal{J}$
is Hamburger determinate if and only if $\dot{J}$ is essentially
self-adjoint \cite{Akhiezer}. If this happens we shall denote by
$J:=\overline{\dot{J}}$ the unique self-adjoint extension of $\dot{J}$.
Then
\begin{equation}
\Dom J=\{\mathbf{f}\in\ell^{2};\,\mathcal{J}\mathbf{f}\in\ell^{2}\}.\label{eq:DomJ}
\end{equation}
The symbol $\varrho(J)$ will stand for the resolvent set of $J$
and the spectral decomposition of $J$ will be written in the form
$J=\int\lambda\,\text{d}E_{\lambda},$ with $E$ being a projection-valued
measure. Then the probability measure on $\mathbb{R}$,
\begin{equation}
\mu(\cdot):=\langle\mathbf{e}_{0},E(\cdot)\mathbf{e}_{0}\rangle,\label{eq:mu_E}
\end{equation}
is the only solution to the Hamburger moment problem. In particular,
$\supp\mu=\Spec J$ and we have
\[
m_{k}=\langle\mathbf{e}_{0},J^{k}\mathbf{e}_{0}\rangle=\int\lambda^{k}\,\text{d}\mu(\lambda),\ k\geq0.
\]

We shall focus on the case when $\dot{J}$ is is essentially self-adjoint
and positive definite, i.e.
\[
\forall\mathbf{f}\in\span\{\mathbf{e}_{n}\},\ \langle\mathbf{f},\mathcal{J}\mathbf{f}\rangle\geq\gamma\|\mathbf{f}\|^{2}
\]
for some $\gamma>0$. Moreover, we shall assume that $J^{-1}$ is
a trace class operator. If so, denote by $\{\lambda_{n};\,n\geq1\}$
the sequence of eigenvalues of $J$ ordered increasingly,
\[
0<\gamma\leq\lambda_{1}<\lambda_{2}<\lambda_{3}<\cdots.
\]
Of course, all eigenvalues $\lambda_{m}$ are simple. Then $\mu$
is supported on the discrete set $\{\lambda_{n};\,n\geq1\}=\Spec J$.
Furthermore, the function $z\mapsto\prod_{n=1}^{\infty}(1-z/\lambda_{n})$
is entire and its zero set coincides with $\Spec J$.

An indispensable companion of $\mathcal{J}$ is an orthonormal polynomial
sequence $\{\hat{P}_{n}(x)\}$ obeying the three-term recurrence equation
\begin{equation}
\alpha_{0}\hat{P}_{1}(x)+(\beta_{0}-x)\hat{P}_{0}(x)=0,\ \alpha_{n}\hat{P}_{n+1}(x)+(\beta_{n}-x)\hat{P}_{n}(x)+\alpha_{n-1}\hat{P}_{n-1}(x)=0,\ n\geq1,\label{eq:recurr_Phat}
\end{equation}
with the initial data $\hat{P}_{0}(x)=1$. Whenever convenient we
shall treat the sequence $\{\hat{P}_{n}(x)\}$ as a column vector
$\hat{P}_{\ast}(x)\in\mathbb{C}^{\infty}$,
\[
\hat{P}_{\ast}(x)^{T}:=\big(\hat{P}_{0}(x),\hat{P}_{1}(x),\hat{P}_{2}(x),\ldots\big).
\]
An analogous notation will be used in case of other polynomial sequences
as well. The recurrence (\ref{eq:recurr_Phat}) in fact means exactly
that $\hat{P}_{\ast}(x)$ is a formal eigenvector of $\mathcal{J}$
corresponding to the eigenvalue $x$,
\begin{equation}
\mathcal{J}\hat{P}_{\ast}(x)=x\hat{P}_{\ast}(x),\label{eq:Ph_eigenv_J}
\end{equation}
with the initial component $\hat{P}_{0}(x)=1$. Recall also the following
relation which is not only easy to verify but also quite useful:
\begin{equation}
\forall n\geq0,\ \hat{P}_{n}(\mathcal{J})\mathbf{e}_{0}=\mathbf{e}_{n}.\label{eq:Phn_J_e0_en}
\end{equation}

If $\mathcal{J}$ is Hamburger determinate then $\{\hat{P}_{n}(x)\}$
is an orthonormal basis in $L^{2}(\mathbb{R},\text{d}\mu(x))$ and
from (\ref{eq:DomJ}) it is seen that $\sum_{n=0}^{\infty}|\hat{P}_{n}(z)|^{2}=\infty$
for all $z\in\mathbb{C}$ unless $z\in\mathbb{R}$ is an eigenvalue
of $J$.

Another useful tool which can be applied in case of a Hamburger determinate
Jacobi matrix $\mathcal{J}$ is that of the sequence of functions
of the second kind \cite{VanAssche}. Let us define an analytic vector-valued
function $\mathbf{w}(z)$ on $\varrho(J)$ with values in $\ell^{2}$
by the equation
\[
\forall z\in\varrho(J),\ \mathbf{w}(z):=(J-z)^{-1}\mathbf{e}_{0}=\int\frac{\hat{P}_{\ast}(\lambda)}{\lambda-z}\,\text{d}\mu(\lambda).
\]
Its components $w_{k}(z)$ are called \emph{functions of the second
kind}. Hence
\[
\forall z\in\varrho(J),\forall k\geq0,\ w_{k}(z)=\int\frac{\hat{P}_{k}(\lambda)}{\lambda-z}\,\text{d}\mu(\lambda).
\]
The first component is also called the \emph{Weyl function} of $J$,
\[
\forall z\in\varrho(J),\ w(z)\equiv w_{0}(z)=\langle\mathbf{e}_{0},(J-z)^{-1}\mathbf{e}_{0}\rangle=\int\frac{\text{d}\mu(\lambda)}{\lambda-z}\,,
\]
and it is, up to a sign, the Stieltjes transform of the measure $\mu$.
Note that, in $\mathbb{C}^{\infty}$, $(\mathcal{J}-z)\mathbf{w}(z)=\mathbf{e}_{0}$
with the initial data $\langle\mathbf{e}_{0},\mathbf{w}(z)\rangle=w(z)$.
This is in fact a recurrence relation which along with the initial
data determines the vector $\mathbf{w}(z)$ unambiguously.

We claim that if $\mathcal{J}$ is Hamburger determinate then
\[
\forall z\in\varrho(J),\forall n\geq0,\ w_{n}(z)\hat{P}_{n}(z)=\langle\mathbf{e}_{n},(J-z)^{-1}\mathbf{e}_{n}\rangle.
\]
Indeed,
\begin{eqnarray*}
w_{n}(z)\hat{P}_{n}(z) & = & \hat{P}_{n}(z)\int_{\mathbb{R}}\frac{\hat{P}_{n}(\lambda)}{\lambda-z}\,\text{d}\mu(\lambda)+\int_{\mathbb{R}}\hat{P}_{n}(\lambda)\,\frac{\hat{P}_{n}(\lambda)-\hat{P}_{n}(z)}{\lambda-z}\,\text{d}\mu(\lambda)\\
 & = & \int\frac{\hat{P}_{n}(\lambda)^{2}}{\lambda-z}\,\text{d}\mu(\lambda)\\
\noalign{\smallskip} & = & \langle\mathbf{e}_{n},(J-z)^{-1}\mathbf{e}_{n}\rangle.
\end{eqnarray*}
In the first equation we have used the orthogonality of the polynomial
sequence and that $\big(\hat{P}_{n}(\lambda)-\hat{P}_{n}(z)\big)/(\lambda-z)$
is a polynomial in $\lambda$ of degree less than $n$. The last equality
follows from (\ref{eq:mu_E}) and (\ref{eq:Phn_J_e0_en}).

Moreover, note that $\int\hat{P}_{n}(\lambda)^{2}\text{d}\mu(\lambda)=1$
and so the function $w_{n}(z)\hat{P}_{n}(z)$ is, up to a sign, the
Stieltjes transform of a probability measure and as such it is known
to have no zeros outside the convex hull of the support of that measure
\cite{Stieltjes}. Hence $w_{n}(z)\hat{P}_{n}(z)$ has no zeros outside
the convex hull of $\supp\mu$. Consequently, if $\mathcal{J}$ is
Hamburger determinate and $J$ is positive definite then $w_{n}(0)\hat{P}_{n}(0)>0$
for all $n\geq0$ and $J^{-1}$ belongs to the trace class if and
only if
\begin{equation}
\Tr J^{-1}=\sum_{n=0}^{\infty}w_{n}(0)\hat{P}_{n}(0)<\infty.\label{eq:Tr_Jinv}
\end{equation}

\begin{theorem}\label{thm:main} Let $\mathcal{J}$ be a Hamburger
determinate Jacobi matrix which is positive definite on $\span\{\mathbf{e}_{n};\,n\geq0\}$
and $J$ be the respective Jacobi matrix operator on $\ell^{2}$.
If $J^{-1}$ is a trace class operator then the series $\sum_{n=0}^{\infty}w_{n}(0)\hat{P}_{n}(z)$
converges locally uniformly in $z\in\mathbb{C}$ and
\begin{equation}
\mathfrak{F}(z):=1-z\sum_{n=0}^{\infty}w_{n}(0)\hat{P}_{n}(z)=\prod_{n=1}^{\infty}\left(1-\frac{z}{\lambda_{n}}\right)\label{eq:Fgoth_series}
\end{equation}
where $\{\lambda_{n};\,n=1,2,3,\ldots\}=\Spec J$.

Particularly this means that the zeros of the entire function $\mathfrak{F}(z)$
are simple and the zero set coincides with the spectrum of $J$. \end{theorem}

\begin{remark} An earlier result derived in \cite{StampachStovicek}
provides another formula for a characteristic function of $J$ written
in a factorized form. This has been done under the assumption
\begin{equation}
\sum_{n=0}^{\infty}\frac{\alpha_{n}^{\,2}}{\beta_{n}\beta_{n+1}}<\infty\ \ \text{and}\ \ \sum_{n=0}^{\infty}\frac{1}{\beta_{n}^{\,2}}<\infty.\label{eq:assum_OaM}
\end{equation}
The both results are in principle independent. The Jacobi matrix treated
as an example below in Section~\ref{sec:AlSalamCarlitz} does not
comply with (\ref{eq:assum_OaM}) and thus the method of \cite{StampachStovicek}
fails in this case. On the other hand, conditions (\ref{eq:assum_OaM})
guarantee only that $J^{-1}$ belongs to the Hilbert-Schmidt class
but not necessarily to the trace class as demonstrated by Example~4.1
in \cite{StampachStovicek}. Therefore this example is not covered
by Theorem~\ref{thm:main}. \end{remark}

The paper is organized as follows. In Section~\ref{sec:preliminaries}
we provide a brief summary of some known results which are substantial
for the proof of Theorem~\ref{thm:main}. They are mostly concerned
with the associated orthogonal polynomials and the functions of the
second kind. Here we also point out an application of the Green function
of a Jacobi matrix which is of importance for our purposes. Section~\ref{sec:the_proof}
is entirely devoted to the proof of Theorem~\ref{thm:main}. In Section~\ref{sec:AlSalamCarlitz}
we treat the Al\textendash Salam\textendash Carlitz II polynomials
as an example of an orthogonal polynomial sequence to which our approach
can be successfully applied.

\section{Preliminaries \label{sec:preliminaries}}

\subsection{Associated polynomials, functions of the second kind and the Weyl
function}

Here we summarize some known results concerning the associated orthogonal
polynomials and the associated functions which are substantial for
the proof of Theorem~\ref{thm:main}. We may provide also a few additional
observations. Further details can be found, for instance, in \cite{Belmehdi,VanAssche}.

Given an arbitrary Jacobi matrix $\mathcal{J}$, the so called sequence
of \emph{polynomials of the second kind}, $\{Q_{n}(x)\}$, is again
determined by the three-term recurrence 
\begin{equation}
\alpha_{n}Q_{n+1}(x)+(\beta_{n}-x)Q_{n}(z)+\alpha_{n-1}Q_{n-1}(x)=0,\ n\geq1,\label{eq:recurr_Q}
\end{equation}
with the initial data $Q_{0}(x)=0$, $Q_{1}(x)=1/\alpha_{0}$. Since
$\{\hat{P}_{n}(z)\}$ and $\{Q_{n}(z)\}$ obey the same recurrence
equation it is a matter of routine manipulations to derive that
\begin{equation}
Q_{n}(z)=\left(\sum_{j=0}^{n-1}\,\frac{1}{\alpha_{j}\hat{P}_{j}(z)\hat{P}_{j+1}(z)}\right)\!\hat{P}_{n}(z),\ n=0,1,2,\ldots,\label{eq:Qn_Pn}
\end{equation}
provided that $\forall n\geq1$, $\hat{P}_{n}(z)\neq0$. Actually,
assuming (\ref{eq:recurr_Phat}) one readily verifies that (\ref{eq:Qn_Pn})
is a solution of (\ref{eq:recurr_Q}) and fulfills the required initial
condition. If $\mathcal{J}$ is Hamburger determinate and $z\in\mathbb{C}$
does not belong to the convex hull of $\supp\mu$ then, indeed, $\hat{P}_{n}(z)\neq0$
for all $n\geq1$.

Note that (\ref{eq:recurr_Phat}) and (\ref{eq:recurr_Q}) can be
respectively rewritten as vector equations in $\mathbb{C}^{\infty}$,
namely
\[
(\mathcal{J}-z)\hat{P}_{\ast}(z)=\mathbf{0},\ (\mathcal{J}-z)Q_{\ast}(z)=\mathbf{e}_{0},
\]
with the initial data $\hat{P}_{0}(z)=1$, $Q_{0}(z)=0$. Consequently,
\[
(\mathcal{J}-z)\big(w(z)\hat{P}_{\ast}(z)+Q_{\ast}(z)\big)=\mathbf{e}_{0},\ \langle\mathbf{e}_{0},w(z)\hat{P}_{\ast}(z)+Q_{\ast}(z)\rangle=w(z).
\]
In the Hamburger determinate case the uniqueness of the solution implies
that
\begin{equation}
\forall z\in\varrho(J),\ w(z)\hat{P}_{\ast}(z)+Q_{\ast}(z)=(J-z)^{-1}\mathbf{e}_{0}=\mathbf{w}(z)\in\ell^{2}.\label{eq:wPh_plus_Q_l2}
\end{equation}
As a matter of fact, note that for any $z\in\varrho(J)$, $w(z)$
is the unique complex number such that $w(z)\hat{P}_{\ast}(z)+Q_{\ast}(z)$
is square summable.

In \cite{Berg}, a generalization of Markov's Theorem has been derived
which is applicable to any Hamburger determinate Jacobi matrix $\mathcal{J}$.
Denote by $\mathcal{Z}_{n}$ the set of roots of $\hat{P}_{n}(x)$
and put
\begin{equation}
\Lambda:=\bigcap_{N=1}^{\infty}\,\overline{\bigcup_{n=N}^{\infty}\mathcal{Z}_{n}}\,.\label{eq:Lambda}
\end{equation}
It is well known that $\supp\mu\subset\Lambda$. Then for any $z\in\mathbb{C}\setminus\Lambda\subset\varrho(J)$
the limit $\lim_{n\to\infty}Q_{n}(z)/\hat{P}_{n}(z)$ exists. In that
case it follows from (\ref{eq:wPh_plus_Q_l2}) that necessarily
\[
\lim_{n\to\infty}\frac{Q_{n}(z)}{\hat{P}_{n}(z)}=-w(z)=-\langle\mathbf{e}_{0},(J-z)^{-1}\mathbf{e}_{0}\rangle=\int\frac{\text{d}\mu(\lambda)}{z-\lambda}\,.
\]

With $\mathcal{J}$ one can associate a sequence of Jacobi matrices
$\mathcal{J}^{(k)}$, $k\geq0$. $\mathcal{J}^{(k)}$ is obtained
from $\mathcal{J}$ by deleting the first $k$ rows and columns. In
particular, $\mathcal{J}^{(0)}=\mathcal{J}$. Every Jacobi matrix
$\mathcal{J}^{(k)}$ again determines an orthonormal polynomial sequence
$\{\hat{P}_{n}^{(k)}(x);\,n\geq0\}$ unambiguously defined by $\mathcal{J}^{(k)}\hat{P}_{\ast}^{(k)}(x)=x\hat{P}_{\ast}^{(k)}(x)$
and $\hat{P}_{0}^{(k)}(x)=1$. Polynomials $\hat{P}_{n}^{(k)}(x)$
are called \emph{associated orthogonal polynomials}. Note that
\begin{equation}
\hat{P}_{n}^{(1)}(x)=Q_{n+1}(x)/Q_{1}(0),\ n\geq0.\label{eq:Phat1_Q}
\end{equation}
If convenient, we shall identify $\hat{P}_{n}(x)\equiv\hat{P}_{n}^{(0)}(x)$.

\begin{proposition} If a Jacobi matrix $\mathcal{J}$ is determinate
then the associated Jacobi matrices $\mathcal{J}^{(k)}$ are determinate
for all $k\geq1$. \end{proposition}

\begin{proof} Obviously it is sufficient to verify only the case
$k=1$. The Weyl function is, up to a sign, the Stieltjes transform
of the probability measure $\mu$ and as such it is known to have
no zeros outside the convex hull of $\supp\mu$ \cite{Stieltjes}.
Taking into account that $w(z)\hat{P}_{\ast}(z)+Q_{\ast}(z)\in\ell^{2}$
but $\hat{P}_{\ast}(z)\notin\ell^{2}$ we deduce that $Q_{\ast}(z)\notin\ell^{2}$,
equivalently $\hat{P}_{\ast}^{(1)}(z)\notin\ell^{2}$, for all $z\in\mathbb{C}\setminus\mathbb{R}$.
Hence the deficiency indices of $\dot{J}^{(1)}$ are $(0,0)$ meaning
that $\dot{J}^{(1)}$ is essentially self-adjoint. \end{proof}

Relation (\ref{eq:Phat1_Q}) can be readily generalized. For a given
$k\geq0$, let a polynomial sequence $\{Q_{n}^{(k)}(x);\,n\geq0\}$
be the unique solution of the equations
\begin{equation}
(\mathcal{J}-x)Q_{\ast}^{(k)}(x)=\mathbf{e}_{k},\ Q_{0}^{(k)}(x)=Q_{1}^{(k)}(x)=\ldots=Q_{k}^{(k)}(x)=0.\label{eq:eqs_Qk}
\end{equation}
Then $Q_{n}^{(0)}(x)\equiv Q_{n}(x)$ and
\begin{equation}
\hat{P}_{n}^{(k)}=Q_{n+k}^{(k-1)}(x)/Q_{k}^{(k-1)}(0),\ n\geq0,\,k\geq1.\label{eq:Phatk_Q}
\end{equation}

Up to the end of the current subsection $\mathcal{J}$ is supposed
to be Hamburger determinate. Again, $J^{(k)}$ is the unique self-adjoint
operator in $\ell^{2}$ corresponding to $\mathcal{J}^{(k)}$. Assume
additionally that
\begin{equation}
J\geq\gamma>0.\label{eq:J_positive}
\end{equation}
Then it is clear that $J^{(k)}\geq\gamma>0$ for all $k\geq1$.

For any $k\geq0$ and $z\in\varrho(J)$ define
\begin{equation}
\mathbf{w}^{(k)}(z):=(J-z)^{-1}\mathbf{e}_{k}\in\ell^{2}.\label{eq:def_vec_wk}
\end{equation}
Particularly, $\mathbf{w}^{(0)}(z)\equiv\mathbf{w}(z)$. The column
vector $\mathbf{w}^{(k)}(z)$ solves the equations
\[
(\mathcal{J}-z)\mathbf{w}^{(k)}(z)=\mathbf{e}_{k}\ \text{in}\ \mathbb{C}^{\infty},\ \langle\mathbf{e}_{0},\mathbf{w}^{(k)}(z)\rangle=w_{k}(z).
\]
By a similar reasoning as above,
\begin{equation}
\forall z\in\varrho(J),\ \mathbf{w}^{(k)}(z)=w_{k}(z)\hat{P}_{\ast}(z)+Q_{\ast}^{(k)}(z)\in\ell^{2}.\label{eq:vec_wk}
\end{equation}
Since $\hat{P}_{\ast}(z)\notin\ell^{2}$ the complex number $w_{k}(z)$
is unique such that $w_{k}(z)\hat{P}_{\ast}(z)+Q_{\ast}^{(k)}(z)$
is square summable.

Referring again to \cite{Berg} we know that the limit $\lim_{n\to\infty}\hat{P}_{n}^{(k+1)}(z)/\hat{P}_{n+1}^{(k)}(z)$
exists for all $k\geq0$ and $z\in\mathbb{C}\setminus[\,\gamma,\infty)$.
In view of (\ref{eq:Phatk_Q}) one finds that the limit $\lim_{n\to\infty}Q_{n}^{(k)}(z)/\hat{P}_{n}(z)$
exists as well. Then, in regard of (\ref{eq:vec_wk}), necessarily
\begin{equation}
\lim_{n\to\infty}\frac{Q_{n}^{(k)}(z)}{\hat{P}_{n}(z)}=-w_{k}(z),\ k\geq0,\,z\in\mathbb{C}\setminus[\,\gamma,\infty).\label{eq:lim_Qkn_over_Phn}
\end{equation}
This is a well known generalization of Markov's theorem, see \cite[Theorem\,2]{VanAssche}.

In view of (\ref{eq:def_vec_wk}), (\ref{eq:vec_wk}) and (\ref{eq:eqs_Qk})
it holds true that
\[
\forall z\in\varrho(J),\,\forall m\leq n,\ \langle\mathbf{e}_{m},(J-z)^{-1}\mathbf{e}_{n}\rangle=\langle\mathbf{e}_{m},w_{n}(z)\hat{P}_{\ast}(z)+Q_{\ast}^{(k)}(z)\rangle=\hat{P}_{m}(z)w_{n}(z).
\]
On the other hand, for $m>n$ we have
\[
\langle\mathbf{e}_{m},(J-z)^{-1}\mathbf{e}_{n}\rangle=\hat{P}_{m}(z)w_{n}(z)+Q_{m}^{(n)}(z)
\]
and, at the same time,
\[
\langle\mathbf{e}_{m},(J-z)^{-1}\mathbf{e}_{n}\rangle=\overline{\langle\mathbf{e}_{n},(J-\overline{z})^{-1}\mathbf{e}_{m}\rangle}=w_{m}(z)\hat{P}_{n}(z).
\]
Whence
\begin{equation}
\forall m>n,\ Q_{m}^{(n)}(z)=w_{m}(z)\hat{P}_{n}(z)-\hat{P}_{m}(z)w_{n}(z).\label{eq:Qmn_wm_Phatn}
\end{equation}
Letting $n=0$ we again get the known equation
\begin{equation}
\forall m\geq1,\ Q_{m}(z)=w_{m}(z)-w(z)\hat{P}_{m}(z).\label{eq:Qm_wm_Phatm}
\end{equation}
Combining (\ref{eq:Qmn_wm_Phatn}) and (\ref{eq:Qm_wm_Phatm}) we
obtain
\[
\forall m>n,\forall z\in\mathbb{C}\setminus[\,\gamma,\infty),\ Q_{m}^{(n)}(z)=Q_{m}(z)\hat{P}_{n}(z)-\hat{P}_{m}(z)Q_{n}(z).
\]
If $J$ is positive definite then this is particularly true for $z=0$.
All these relations are well known, see \cite{VanAssche}.

(\ref{eq:J_positive}) is still supposed to be true. It is well known
that all roots of the polynomials $\hat{P}_{m}^{(n)}(x)$ and $Q_{m}^{(n)}(x)$
lie in $[\,\gamma,\infty)$ and therefore $\Lambda$, as defined in
(\ref{eq:Lambda}), is a subset of $[\,\gamma,\infty)$. From (\ref{eq:Qn_Pn})
it follows that for $z\in\mathbb{C}\setminus[\,\gamma,\infty)$, 
\[
w(z)=-\lim_{n\to\infty}\frac{Q_{n}(z)}{\hat{P}_{n}(z)}=-\sum_{j=0}^{\infty}\,\frac{1}{\alpha_{j}\hat{P}_{j}(z)\hat{P}_{j+1}(z)}
\]
and 
\begin{equation}
w_{n}(z)=w(z)\hat{P}_{n}(z)+Q_{n}(z)=-\left(\sum_{j=n}^{\infty}\,\frac{1}{\alpha_{j}\hat{P}_{j}(z)\hat{P}_{j+1}(z)}\right)\!\hat{P}_{n}(z),\ n\geq0.\label{eq:wn_series}
\end{equation}

Furthermore the polynomials $\hat{P}_{n}(x)$ and $Q_{n}(x)$ have
all their leading coefficients positive and since they do not change
the sign on $(-\infty,\gamma)$ we have
\[
\sgn\hat{P}_{n}(0)=(-1)^{n},\ n\geq0;\ \sgn Q_{n}(0)=(-1)^{n+1},\ n\geq1.
\]
Moreover, from (\ref{eq:wn_series}) it is seen that $w_{n}(0)/\hat{P}_{n}(0)>0$
for all $n\geq0$.

\subsection{A relation between an orthogonal polynomial sequence and the respective
Green function}

Given a Jacobi matrix $\mathcal{J}$ (not necessarily Hamburger determinate)
let $\mathcal{G}$ be a strictly lower triangular matrix with the
entries
\begin{equation}
\mathcal{G}_{m,n}=Q_{m}^{(n)}(0)\ \ \text{for}\ m,n\geq0,\ \mathcal{G}_{m,n}=0\ \ \text{otherwise},\label{eq:G}
\end{equation}
with $\{Q_{m}^{(n)}(x)\}$ being defined in (\ref{eq:eqs_Qk}). Then
$\mathcal{G}$ is a right inverse of $\mathcal{J}$ in $\mathbb{C}^{\infty}$,
$\mathcal{J}\mathcal{G}=I$, and $\mathcal{G}$ is obviously the unique
strictly lower triangular matrix with this property. $\mathcal{G}$
can be interpreted as the \emph{Green function} of the Jacobi matrix
$\mathcal{J}$ \cite{Geronimo}.

\begin{theorem}\label{thm:L} Let $\mathcal{J}$ be a Jacobi matrix,
$\{\hat{P}_{n}(x)\}$ be the respective orthonormal polynomial sequence
and $\mathcal{G}$ be the Green function of $\mathcal{J}$. Then
\begin{equation}
\hat{P}_{\ast}(x)=(I-x\mathcal{G})^{-1}\hat{P}_{\ast}(0).\label{eq:Phat_G}
\end{equation}
Conversely, equation (\ref{eq:Phat_G}) determines the strictly lower
triangular matrix $\mathcal{G}$ unambiguously. \end{theorem}

\begin{lemma}\label{thm:MP_eq_0} If $\mathcal{M}\in\mathbb{C}^{\infty,\infty}$
is a lower triangular matrix and $\mathcal{M}\hat{P}_{\ast}(x)=0$
holds in $\mathbb{C}[x]^{\infty}$ then $\mathcal{M}=0$. \end{lemma}

\begin{proof} Since for every $n$, $\hat{P}_{n}(x)$ is a polynomial
of degree $n$ with with a nonzero leading coefficient  there exists
a unique lower triangular matrix $\Pi$ with no zeros on the diagonal
such that
\[
\hat{P}_{\ast}(x)=\Pi\mathbf{v}(x)\ \ \text{where}\ \ \mathbf{v}(x)^{T}=(1,x,x^{2},\ldots).
\]
Then $\mathcal{M}\hat{P}_{\ast}(x)=0$ is equivalent to $\mathcal{M}\Pi=0$.
But $\Pi$ is invertible (in $\mathbb{C}^{\infty}$) and therefore
$\mathcal{M}=0$. \end{proof}

\begin{proof}[Proof of Theorem \ref{thm:L}] Owing to (\ref{eq:Ph_eigenv_J})
and since $\mathcal{J}\mathcal{G}=I$ we have $\mathcal{J}(I-x\mathcal{G})\hat{P}_{\ast}(x)=0$.
Obviously, if a column vector $\mathbf{v}$, $\mathbf{v}^{T}=(v_{0},v_{1,},v_{2},\ldots)$,
obeys $\mathcal{J}\mathbf{v}=0$ and $v_{0}=0$ then $\mathbf{v}=0$.
In particular, this observation is applicable to
\[
\mathbf{v}=(I-x\mathcal{G})\hat{P}_{\ast}(x)-\hat{P}_{\ast}(0),
\]
and (\ref{eq:Phat_G}) follows.

Conversely, suppose $\hat{P}_{\ast}(x)=(I-x\mathcal{L})^{-1}\hat{P}_{\ast}(0)$
holds for a strictly lower triangular matrix $\mathcal{L}$. Then
$\hat{P}_{\ast}(x)=\hat{P}_{\ast}(0)+x\mathcal{L}\hat{P}_{\ast}(x).$
In regard of (\ref{eq:Ph_eigenv_J}) we obtain
\[
x\hat{P}_{\ast}(x)=\mathcal{J}\hat{P}_{\ast}(x)=\mathcal{J}\big(\hat{P}_{\ast}(0)+x\mathcal{L}\hat{P}_{\ast}(x)\big)=x\mathcal{J}\mathcal{L}\hat{P}_{\ast}(x)
\]
whence $(I-\mathcal{J}\mathcal{L})\hat{P}_{\ast}(x)=0$. Observing
that $I-\mathcal{J}\mathcal{L}$ is lower triangular we have, by Lemma~\ref{thm:MP_eq_0},
$\mathcal{J}\mathcal{L}=I$. But this equation determines $\mathcal{L}$
unambiguously. \end{proof}

Finally, note that in the case $\mathcal{J}$ is Hamburger determinate
and $J$ is positive definite (\ref{eq:lim_Qkn_over_Phn}) tells us
that
\begin{equation}
\lim_{m\to\infty}\frac{\mathcal{G}_{m,n}}{\hat{P}_{m}(0)}=-w_{n}(0)=-\langle\mathbf{e}_{n},J^{-1}\mathbf{e}_{0}\rangle.\label{eq:lim_Lhat_Phat}
\end{equation}

\section{Proof of the main theorem \label{sec:the_proof}}

\begin{proof}[Proof of Theorem~\ref{thm:main}] First, we shall show
that the series in (\ref{eq:Fgoth_series}) converges locally uniformly
on $\mathbb{C}$ and, moreover, $\mathfrak{F}(z)$ is a locally uniform
limit of the sequence $\{\hat{P}_{n}(z)/\hat{P}_{n}(0)\}$. Let
\[
\kappa_{n}:=w_{n}(0)\hat{P}_{n}(0).
\]
By the assumptions, see (\ref{eq:Tr_Jinv}), $\kappa_{n}$ are all
positive and $\{\kappa_{n}\}\in\ell^{1}$. Equation (\ref{eq:Phat_G})
means that
\begin{equation}
\forall n\geq0,\ \hat{P}_{n}(z)=\hat{P}_{n}(0)+\sum_{\ell=1}^{n}z^{\ell}\sum_{0\leq k_{1}<k_{2}<\ldots<k_{\ell}<n}\mathcal{G}_{nk_{\ell}}\mathcal{G}_{k_{\ell}k_{\ell-1}}\cdots\mathcal{G}_{k_{2}k_{1}}\hat{P}_{k_{1}}(0).\label{eq:Phat_sum}
\end{equation}
Let us introduce, for the purpose of the proof, a strictly lower triangular
matrix $K$ and a vector $\mathbf{k}\in\ell^{2}$:
\begin{equation}
\forall m,n\geq0,\ K_{m,n}:=\frac{\hat{P}_{n}(0)}{\hat{P}_{m}(0)}\,\sqrt{\frac{\kappa_{m}}{\kappa_{n}}}\,\mathcal{G}_{m,n};\ \mathbf{k}^{T}:=(\sqrt{\kappa_{0}},\sqrt{\kappa_{1}},\sqrt{\kappa_{2}},\ldots).\label{eq:K_t}
\end{equation}

Let us show that
\begin{equation}
\forall m>n,\ \left|\frac{\hat{P}_{n}(0)}{\hat{P}_{m}(0)}\,\mathcal{G}_{m,n}\right|\leq\kappa_{n}\ \,\text{and}\,\ |K_{m,n}|\leq\sqrt{\kappa_{m}\kappa_{n}}\,.\label{eq:estim_K}
\end{equation}
In fact, the latter estimate is obviously a consequence of the former
one. Concerning the former estimate we have, in view of (\ref{eq:G})
and (\ref{eq:Qmn_wm_Phatn}),
\[
\left|\frac{\hat{P}_{n}(0)}{\hat{P}_{m}(0)}\,\mathcal{G}_{m,n}\right|=\hat{P}_{n}(0)^{2}\left|\frac{w_{m}(0)}{\hat{P}_{m}(0)}-\frac{w_{n}(0)}{\hat{P}_{n}(0)}\right|\!.
\]
Remember that
\[
\frac{w_{n}(0)}{\hat{P}_{n}(0)}=-\sum_{j=n}^{\infty}\,\frac{1}{\alpha_{j}\hat{P}_{j}(0)\hat{P}_{j+1}(0)}>0
\]
is a decreasing sequence in $n$, see (\ref{eq:wn_series}). Hence,
still assuming that $m>n$,
\[
\left|\frac{\hat{P}_{n}(0)}{\hat{P}_{m}(0)}\,\mathcal{G}_{m,n}\right|\leq\hat{P}_{n}(0)^{2}\max\!\left\{ \frac{w_{m}(0)}{\hat{P}_{m}(0)}\,,\frac{w_{n}(0)}{\hat{P}_{n}(0)}\right\} \leq\hat{P}_{n}(0)^{2}\,\frac{w_{n}(0)}{\hat{P}_{n}(0)}=\kappa_{n}.
\]
The estimate in particular means that the matrix $K$ can be regarded
as a Hilbert-Schmidt operator on $\ell^{2}$.

We can rewrite (\ref{eq:Phat_sum}) as
\begin{equation}
\frac{\hat{P}_{n}(z)}{\hat{P}_{n}(0)}=1+\sum_{j=1}^{n}\frac{z^{j}}{\sqrt{\kappa_{n}}}\,\langle\mathbf{e}_{n},K^{j}\mathbf{k}\rangle.\label{eq:Phat/Phat0_series}
\end{equation}
For all $j$, $n$, $1\leq j\leq n$, we can estimate
\begin{eqnarray}
\left|\frac{1}{\sqrt{\kappa_{n}}}\,\langle\mathbf{e}_{n},K^{j}\mathbf{k}\rangle\right| & = & \Bigg|\sum_{0\leq k_{1}<k_{2}<\ldots<k_{j}<n}\frac{1}{\sqrt{\kappa_{n}}}\,K_{nk_{j}}K_{k_{j}k_{j-1}}\cdots K_{k_{2}k_{1}}\sqrt{\kappa_{k_{1}}}\Bigg|\nonumber \\
 & \leq & \sum_{0\leq k_{1}<k_{2}<\ldots<k_{j}<n}\kappa_{k_{j}}\kappa_{k_{j-1}}\cdots\kappa_{k_{2}}\kappa_{k_{1}}\nonumber \\
 & \leq & \frac{1}{j!}\Bigg(\sum_{s=1}^{\infty}\kappa_{s}\Bigg)^{\!j}.\label{eq:estim_Kj_t}
\end{eqnarray}
Consequently,
\begin{equation}
\forall j\geq0,\ \|K^{j}\mathbf{k}\|\leq\frac{1}{j!}\Bigg(\sum_{s=1}^{\infty}\kappa_{s}\Bigg)^{\!j+1/2}.\label{eq:estim_norm_Kj_t}
\end{equation}

Very similarly, for $m>n$ and $j\geq1$,
\[
\langle\mathbf{e}_{m},K^{j}\mathbf{e}_{n}\rangle|\leq\sqrt{\kappa_{m}\kappa_{n}}\,\sum_{n<k_{1}<k_{2}<\ldots<k_{j-1}<m}\kappa_{k_{j-1}}\cdots\kappa_{k_{2}}\kappa_{k_{1}}\leq\frac{\sqrt{\kappa_{m}\kappa_{n}}}{(j-1)!}\Bigg(\sum_{s=1}^{\infty}\kappa_{s}\Bigg)^{\!j-1}\!.
\]
It follows that
\[
\forall j\geq1,\ \|K^{j}\|\leq\frac{1}{2(j-1)!}\Bigg(\sum_{s=1}^{\infty}\kappa_{s}\Bigg)^{\!j}.
\]
Hence the spectral radius of $K$ equals $0$. Estimate (\ref{eq:estim_Kj_t})
implies that, in addition,
\begin{equation}
\forall z\in\mathbb{C},\,\forall n\geq0,\ |\langle\mathbf{e}_{n},(I-zK)^{-1}\mathbf{k}\rangle|\leq\exp\!\left(|z|\sum_{s=1}^{\infty}\kappa_{s}\right)\!\sqrt{\kappa_{n}}\,.\label{eq:estim_Kresolv_t}
\end{equation}

Furthermore, (\ref{eq:lim_Lhat_Phat}) means that
\[
\lim_{n\to\infty}\frac{K_{n,k}}{\sqrt{\kappa_{n}}}=-\frac{w_{k}(0)\hat{P}_{k}(0)}{\sqrt{\kappa_{k}}}\,,
\]
and we have
\[
\frac{1}{\sqrt{\kappa_{n}}}\,\langle\mathbf{e}_{n},K^{j}\mathbf{k}\rangle=\sum_{k=0}^{\infty}\frac{K_{n,k}}{\sqrt{\kappa_{n}}}\,\langle\mathbf{e}_{k},K^{j-1}\mathbf{k}\rangle
\]
along with the estimate, as it follows from (\ref{eq:estim_Kj_t}),
\[
\left|\frac{K_{n,k}}{\sqrt{\kappa_{n}}}\,\langle\mathbf{e}_{k},K^{j-1}\mathbf{k}\rangle\right|\leq\frac{\kappa_{k}}{(j-1)!}\Bigg(\sum_{s=1}^{\infty}\kappa_{s}\Bigg)^{\!j-1}.
\]
Since this upper bound is independent of $n$ and summable in $k$,
\[
\lim_{n\to\infty}\frac{1}{\sqrt{\kappa_{n}}}\,\langle\mathbf{e}_{n},K^{j}\mathbf{k}\rangle=-\sum_{k=0}^{\infty}\frac{w_{k}(0)\hat{P}_{k}(0)}{\sqrt{\kappa_{k}}}\,\langle\mathbf{e}_{k},K^{j-1}\mathbf{k}\rangle=-\langle\mathbf{k},K^{j-1}\mathbf{k}\rangle.
\]

Similarly, the upper bound in (\ref{eq:estim_Kj_t}) is independent
of $n$ and we obtain from (\ref{eq:Phat/Phat0_series})
\[
\lim_{n\to\infty}\frac{\hat{P}_{n}(z)}{\hat{P}_{n}(0)}=1-z\sum_{j=0}^{\infty}z^{j}\langle\mathbf{k},K^{j}\mathbf{k}\rangle=1-z\,\langle\mathbf{k},(I-zK)^{-1}\mathbf{k}\rangle.
\]
The estimate guarantees that the limit is locally uniform on $\mathbb{C}$
and, in view of (\ref{eq:estim_norm_Kj_t}), the resulting series,
too, converges uniformly on every compact subset of $\mathbb{C}$.

From defining equations (\ref{eq:K_t}) along with (\ref{eq:Phat_G})
one finds that
\[
\sqrt{\kappa_{n}}\,\langle\mathbf{e}_{n},(I-zK)^{-1}\mathbf{k}\rangle=w_{n}(0)\hat{P}_{n}(z).
\]
Hence
\[
1-z\,\langle\mathbf{k},(I-zK)^{-1}\mathbf{k}\rangle=1-z\sum_{n=0}^{\infty}w_{n}(0)\hat{P}_{n}(z).
\]
Moreover, referring to (\ref{eq:estim_Kresolv_t}) we have the bound
\begin{equation}
|w_{n}(0)\hat{P}_{n}(z)|=|\sqrt{\kappa_{n}}\,\langle\mathbf{e}_{n},(I-zK)^{-1}\mathbf{k}\rangle|\leq\exp\!\left(|z|\sum_{s=1}^{\infty}\kappa_{s}\right)\!\kappa_{n}\label{eq:estim_wn_Phn}
\end{equation}
showing that the series converges locally uniformly on $\mathbb{C}$.

Second, we shall show that $\mathfrak{F}^{-1}(0)=\Spec J$ with all
zeros of $\mathfrak{F}(z)$ being simple. We assume that $J$ is positive
definite, and we have already proved that $J^{-1}$ is a compact operator.
Hence $\Spec J$ constitutes of simple eigenvalues which can be ordered
increasingly. We shall use the notation
\begin{equation}
\Spec J=\{\lambda_{n};\,n\geq1\},\ 0<\lambda_{1}<\lambda_{2}<\lambda_{3}<\ldots,\label{eq:specJ}
\end{equation}
and we have $\lim\lambda_{n}=+\infty$.

As explained for instance in \cite[\S II.4]{Chihara}, each eigenvalue
$\lambda_{j}$ can be obtained as the limit of a sequence of roots
of the polynomials $\hat{P}_{n}(x)$. Denote by $x_{n,j}$ the $j$th
root of the polynomial $\hat{P}_{n}(x)$, $n\geq1$, $j=1,2,\ldots,n$,
while supposing that the roots are numbered in increasing order. Then,
for a fixed $j\in\mathbb{N}$, the sequence $\{x_{n,j};\,n=j,j+1,j+2,\ldots\}$
is strictly decreasing and
\begin{equation}
\lim_{n\to\infty}x_{n,j}=\lambda_{j}.\label{eq:lim_xnj_lmbd}
\end{equation}
Using the already proven convergence, $\hat{P}_{n}(z)/\hat{P}_{n}(0)\rightrightarrows\mathfrak{F}(z)$
locally on $\mathbb{C}$, one can determine the zero set $\mathcal{F}^{-1}(0)$.
It is a basic fact that $\hat{P}_{n}'(z)/\hat{P}_{n}(0)\rightrightarrows\mathfrak{F}'(z)$
locally on $\mathbb{C}$, and since the entire function $\mathfrak{F}(z)$
is not equal to zero identically its zero set has no limit points.

Consider a positively oriented circle $C(\lambda_{j},\epsilon)$ centered
at some eigenvalue $\lambda_{j}$ and with a radius $\epsilon>0$
sufficiently small so that the distance of $\lambda_{j}$ from the
rest of the spectrum of $J$ is greater than $\epsilon$ and such
that $\mathfrak{F}(z)$ has no zeros on $C(\lambda_{j},\epsilon)$.
In view of (\ref{eq:lim_xnj_lmbd}), $\hat{P}_{n}(z)$ has exactly
one root in the interior of $C(\lambda_{j},\epsilon)$ for all sufficiently
large $n$. Then the number of zeros of $\mathfrak{F}(z)$ located
in the interior of $C(\lambda_{j},\epsilon)$ equals
\[
\frac{1}{2\pi i}\int_{C(\lambda_{j},\epsilon)}\frac{\mathfrak{F}'(z)}{\mathfrak{F}(z)}\,\mbox{d}z=\lim_{n\to\infty}\,\frac{1}{2\pi i}\int_{C(\lambda_{j},\epsilon)}\frac{\hat{P}_{n}'(z)}{\hat{P}_{n}(z)}\,\mbox{d}z=1.
\]
Since $\epsilon$ can be made arbitrarily small it is clear that $\lambda_{j}$
must be a simple root of $\mathfrak{F}(z)$. Moreover, by a very analogous
reasoning, if $z_{0}\in\mathbb{C}$ is not an eigenvalue of $J$ then
it cannot be a root of $\mathfrak{F}(z)$.

Third, let us verify (\ref{eq:Fgoth_series}). With the above estimates
it is immediate to show that the order of $\mathfrak{F}$ does not
exceed one. Actually, from (\ref{eq:Fgoth_series}) and (\ref{eq:estim_wn_Phn})
it is seen that
\[
|\mathfrak{F}(z)|\leq1+|z|\exp\!\left(|z|\sum_{j=0}^{\infty}\kappa_{j}\right)\sum_{j=0}^{\infty}\kappa_{j}\leq\exp(C|z|)
\]
for some constant $C>0$. The assertion readily follows. Note that
by Hadamard's Factorization Theorem and since $\mathfrak{F}(0)=1$
there exists $c\in\mathbb{R}$ such that
\[
\mathfrak{F}(z)=e^{cz}\prod_{n=1}^{\infty}\left(1-\frac{z}{\lambda_{n}}\right)\!e^{z/\lambda_{n}}.
\]
But by our assumptions $\Tr J^{-1}=\sum_{n=1}^{\infty}\lambda_{n}^{\,-1}<\infty$
and therefore we even have
\begin{equation}
\mathfrak{F}(z)=e^{dz}\prod_{n=1}^{\infty}\left(1-\frac{z}{\lambda_{n}}\right)\label{eq:Hadamard0}
\end{equation}
where $d=c+\sum_{n=1}^{\infty}\lambda_{n}^{\,-1}$. To determine $d$
one can calculate $\mathfrak{F}'(0)$ using (\ref{eq:Fgoth_series})
and (\ref{eq:Hadamard0}),
\[
\mathfrak{F}'(0)=-\sum_{n=0}^{\infty}w_{n}(0)\hat{P}_{n}(0)=d-\sum_{n=1}^{\infty}\frac{1}{\lambda_{n}}\,.
\]
According to (\ref{eq:Tr_Jinv}),
\[
\sum_{n=0}^{\infty}w_{n}(0)\hat{P}_{n}(0)=\Tr J^{-1}=\sum_{n=1}^{\infty}\frac{1}{\lambda_{n}}\,.
\]
Hence $d=0$. \end{proof}

\begin{remark} Comparing (\ref{eq:Lambda}), where we let $\mathcal{Z}_{n}=\{x_{n,j};\,j=1,2,\ldots,n\}$,
to the description of the spectrum of $J$ given in (\ref{eq:specJ}),
(\ref{eq:lim_xnj_lmbd}) we find that in this case $\Lambda=\Spec J$.
\end{remark}

\section{An example: the Al-Salam\textendash Carlitz II polynomials \label{sec:AlSalamCarlitz}}

Everywhere in what follows, $0<q<1$. The Jacobi matrix $J$ corresponding
to the Al-Salam\textendash Carlitz II polynomials \cite{AlSalamCarlitz}
has the entries ($\alpha_{n}>0$)
\[
\alpha_{n}^{\,2}=aq^{-2n-1}(1-q^{n+1}),\ \beta_{n}=(a+1)q^{-n},
\]
where $a>0$ is a parameter.

It is straightforward to verify that
\[
J=WW^{T}+I
\]
where $W$ is a lower triangular matrix with the entries
\[
W_{n,n}=\sqrt{a}\,q^{-n/2},\ W_{n+1,n}=q^{-(n+1)/2}\sqrt{1-q^{n+1}}\,,\ n=0,1,2,\ldots,
\]
and $W_{m,n}=0$ otherwise. Hence $J\geq1$ on $\span\{\mathbf{e}_{n};\,n=0,1,2,\ldots\}$.

The Al-Salam\textendash Carlitz II polynomials can be expressed as
\cite{KoekoekLeskySwarttouw}
\begin{eqnarray}
V_{n}^{(a)}(x;q) & = & (-a)^{n}q^{-n(n-1)/2}\,\,_{2}\phi_{0}\!\left(q^{-n},x;-;q,\frac{q^{n}}{a}\right)\nonumber \\
 & = & (-a)^{n}q^{-n(n-1)/2}(q;q)_{n}\sum_{k=0}^{n}\frac{(x;q)_{k}a^{-k}}{(q;q)_{n-k}(q;q)_{k}}\,,\ n\in\mathbb{Z}_{+}.\label{eq:Vn_hyper}
\end{eqnarray}
The respective orthonormal polynomial sequence takes the form
\begin{equation}
\hat{P}_{n}(x)=\frac{q^{n^{2}/2}a^{-n/2}}{\sqrt{(q;q)_{n}}}\,V_{n}^{(a)}(x;q),\ n\in\mathbb{Z}_{+}.\label{eq:Phatn_Van}
\end{equation}

Values at the point $x=a$ have particularly simple form. Using the
identity \cite[Eq.\,(1.11.7)]{KoekoekLeskySwarttouw}
\[
\,_{2}\phi_{0}\!\left(q^{-n},a;-;q,\frac{q^{n}}{a}\right)=a^{-n},\ n\in\mathbb{Z}_{+},
\]
we obtain
\begin{equation}
\hat{P}_{n}(a)=\frac{(-1)^{n}}{\sqrt{(q;q)_{n}}}\left(\frac{q}{a}\right)^{\!n/2}.\label{eq:Pn_a}
\end{equation}
For the orthogonal polynomials of the second kind (\ref{eq:Qn_Pn})
we get the value
\[
Q_{n}(a)=\left(\sum_{j=0}^{n-1}\,\frac{1}{\alpha_{j}\hat{P}_{j}(a)\hat{P}_{j+1}(a)}\right)\!\hat{P}_{n}(a)=\frac{(-1)^{n+1}}{\sqrt{(q;q)_{n}}}\left(\frac{q}{a}\right)^{\!n/2}\sum_{j=0}^{n-1}\,(q;q)_{j}\,a^{j}.
\]
Note that
\begin{equation}
C_{1}(qa)^{n}\leq|Q_{n}(a)|\leq C_{2}\left(\frac{q}{a}\right)^{\!n/2}\frac{a^{n}-1}{a-1}\label{eq:Qn_a_estim}
\end{equation}
where the constants $C_{1}$, $C_{2}$ do not depend on $n$.

With the additional assumption $a<1$ we can also evaluate functions
of the second kind at $x=a$. In particular the value of the Weyl
function equals
\[
w(a)=w_{0}(a)=-\lim_{n\to\infty}\frac{Q_{n}(a)}{\hat{P}_{n}(a)}=\sum_{j=0}^{\infty}(q;q)_{j}a^{j}.
\]
For a general index $n$ we have (\ref{eq:Qm_wm_Phatm})
\begin{equation}
w_{n}(a)=w(a)\hat{P}_{n}(a)+Q_{n}(a)=\frac{(-1)^{n}}{\sqrt{(q;q)_{n}}}\left(\frac{q}{a}\right)^{\!n/2}\sum_{j=n}^{\infty}\,(q;q)_{j}\,a^{j}.\label{eq:wn_a}
\end{equation}

From (\ref{eq:Pn_a}), (\ref{eq:Qn_a_estim}) it is immediately seen,
as is well known \cite{Chihara,BergValent}, that the Hamburger moment
problem is indeterminate if and only if $q<a<q^{-1}$. To this end,
it suffices to recall the classical result \cite{Akhiezer} asserting
that a Hamburger moment problem is indeterminate if and only if the
sequences $\{\hat{P}_{n}(w)\}$ and $\{Q_{n}(w)\}$ are both square
summable where $w$ is an arbitrary fixed point in $\mathbb{C}$.

From now on we assume that $0<a\leq q$. Hence the Hamburger moment
problem is determinate. We wish to apply our method to evaluate the
characteristic function $\mathfrak{F}(z)$ for the Al-Salam\textendash Carlitz
II polynomials. This goal can be achieved if we shift the Jacobi matrix
by a constant considering $J-aI$ instead of $J$. The orthonormal
polynomial sequence corresponding to the shifted matrix is $\{\hat{P}_{n}(z+a\}$
and, similarly, the sequence of functions of the second kind after
the shift has the form $\{w_{n}(z+a)\}$.

It is straightforward to verify the assumptions of Theorem~\ref{thm:main}
for the shifted Jacobi matrix. First note that by our assumptions
$0<a\leq q<1$ and hence $J-a$ is still positive definite on $\span\{\mathbf{e}_{n};\,n=0,1,2,\ldots\}$.
Referring to (\ref{eq:Pn_a}), (\ref{eq:wn_a}) we have
\[
\sum_{n=0}^{\infty}w_{n}(a)\hat{P}_{n}(a)=\sum_{n=0}^{\infty}\frac{1}{(q;q)_{n}}\left(\frac{q}{a}\right)^{\!n}\sum_{j=n}^{\infty}\,(q;q)_{j}\,a^{j}\leq\frac{1}{(q;q)_{\infty}}\sum_{n=0}^{\infty}\left(\frac{q}{a}\right)^{\!n}\frac{a^{n}}{1-a}<\infty.
\]
Consequently, $(J-a)^{-1}$ belongs even to the trace class.

Put
\[
\mathfrak{F}^{(a)}(z):=1-z\sum_{n=0}^{\infty}w_{n}(a)\hat{P}_{n}(z+a).
\]
In view of (\ref{eq:Fgoth_series}) we have
\[
\mathfrak{F}^{(a)}(z)=\prod_{n=1}^{\infty}\!\left(1-\frac{z}{\lambda_{n}-a}\right)
\]
whence
\[
\mathfrak{F}^{(a)}(z-a)=1-(z-a)\sum_{n=0}^{\infty}w_{n}(a)\hat{P}_{n}(z)=\prod_{n=1}^{\infty}\!\left(1-\frac{z}{\lambda_{n}}\right)\!\Big/\prod_{n=1}^{\infty}\!\left(1-\frac{a}{\lambda_{n}}\right).
\]

\begin{proposition}\label{thm:Fz_AlSalamCarlitz} For $0<a\leq q$
we have
\[
1-(z-a)\sum_{n=0}^{\infty}w_{n}(a)\hat{P}_{n}(z)=\frac{(z;q)_{\infty}}{(a;q)_{\infty}}\,.
\]
Consequently, $\prod_{n=1}^{\infty}(1-z/\lambda_{n})=(z;q)_{\infty}$.
\end{proposition}

\begin{lemma} For all $a,c\in\mathbb{C}$, $c\neq1,q^{-1},q^{-2},\dots$,
it holds true that
\begin{equation}
\,_{2}\phi_{1}(a,q;qc;q,q)=\frac{1-c}{a-c}\!\left(1-\frac{(a;q)_{\infty}}{(c;q)_{\infty}}\right)\!.\label{eq:2phi1_lemma}
\end{equation}
\end{lemma}

\begin{proof} Let
\[
F(a,c)=1-\frac{a-c}{1-c}\,\,_{2}\phi_{1}(a,q;qc;q,q).
\]
We have to show that $F(a,c)=(a;q)_{\infty}/(c;q)_{\infty}$. Since
\[
\,_{2}\phi_{1}(a,q;b;q,z)=1+\frac{(1-a)z}{1-b}\,\,_{2}\phi_{1}(qa,q;qb;q,q)
\]
one readily verifies that $F(a,c)=\big((1-a)/(1-c)\big)F(qa,qc)$
and, by mathematical induction,
\[
\forall n\in\mathbb{Z}_{+},\,F(a,c)=\frac{(a;q)_{n}}{(c;q)_{n}}\,F(q^{n}a,q^{n}c).
\]
Taking the limit $n\to\infty$ gives the result. \end{proof}

\begin{proof}[Proof of Proposition~\ref{thm:Fz_AlSalamCarlitz}] In
view of (\ref{eq:Phatn_Van}) we can write
\[
\sum_{n=0}^{\infty}w_{n}(a)\hat{P}_{n}(z)=\sum_{j=0}^{\infty}X_{j}a^{j}\ \ \text{where}\ \ X_{j}:=\sum_{n=0}^{\infty}(-1)^{n}q^{n(n+1)/2}\,\frac{(q;q)_{j+n}}{(q;q)_{n}}\,V_{n}^{(a)}(z;q).
\]
Using (\ref{eq:Vn_hyper}) we can express
\begin{eqnarray*}
X_{j} & = & \sum_{n=0}^{\infty}(qa)^{n}(q;q)_{j+n}\sum_{k=0}^{n}\frac{(z;q)_{k}a^{-k}}{(q;q)_{n-k}(q;q)_{k}}\\
 & = & \sum_{k=0}^{\infty}\frac{(z;q)_{k}(q;q)_{j+k}\,q^{k}}{(q;q)_{k}}\,\sum_{\ell=0}^{\infty}\frac{(q^{j+k+1};q)_{\ell}}{(q;q)_{\ell}}\,(qa)^{\ell}\\
 & = & \sum_{k=0}^{\infty}\frac{(z;q)_{k}(q;q)_{j+k}\,q^{k}}{(q;q)_{k}(qa;q)_{j+k+1}}\,.
\end{eqnarray*}
In the last step we have used the $q$-Binomial Theorem \cite[Eq.\,(II.3)]{GasperRahman}
\[
\sum_{\ell=0}^{\infty}\frac{(u;q)_{\ell}}{(q;q)_{\ell}}\,z^{\ell}=\frac{(uz;q)_{\infty}}{(z;q)_{\infty}}\,,\ |z|<1.
\]
Now an application of the $q$-Gauss summation \cite[Eq.\,(II.7)]{GasperRahman}
\[
\,_{2}\phi_{1}\!\left(a,b;c;q,\frac{c}{ab}\right)=\frac{(c/a;q)_{\infty}(c/b;q)_{\infty}}{(c;q)_{\infty}\big(c/(ab);q\big)_{\infty}}
\]
gives
\[
\sum_{j=0}^{\infty}\frac{(q^{k+1};q)_{j}}{(q^{k+2}a;q)_{j}}\,a^{j}=\frac{1-q^{k+1}a}{1-a}
\]
and therefore
\begin{equation}
\sum_{n=0}^{\infty}w_{n}(a)\hat{P}_{n}(z)=\sum_{j=0}^{\infty}X_{j}a^{j}=\frac{1}{1-a}\sum_{k=0}^{\infty}\frac{(z;q)_{k}}{(qa;q)_{k}}\,q^{k}=\frac{\,_{2}\phi_{1}(z,q;qa;q,q)}{1-a}\,.\label{eq:sum_wnPhatn}
\end{equation}
Combining (\ref{eq:sum_wnPhatn}) and (\ref{eq:2phi1_lemma}) we obtain
the desired equation. \end{proof}

Referring to Proposition~\ref{thm:Fz_AlSalamCarlitz} we conclude
that an application of Theorem~\ref{thm:main} in this example shows
that the zero set of the function $z\mapsto(z;q)_{\infty}$ coincides
with the spectrum of $J$ and, at the same time, with the support
of the orthogonality measure for the Al-Salam\textendash Carlitz II
polynomials. This is actually the correct answer as the orthogonality
measure in this case is known to be supported on the set $\{q^{-n};\,n\in\mathbb{Z}_{+}\}$
\cite{AlSalamCarlitz,KoekoekLeskySwarttouw}.

\section*{Acknowledgments}

The author acknowledges gratefully partial support by the Ministry
of Education, Youth and Sports of the Czech Republic project no. CZ.02.1.01/0.0/0.0/16\_019/0000778.

\end{document}